\documentclass[a4paper,10pt]{amsart}
\usepackage{amsmath,amsthm,amssymb,latexsym,enumerate,xcolor,hyperref}
\usepackage{graphicx}
\usepackage[margin=2.8cm]{geometry}

\numberwithin{equation}{section}

\usepackage{comment}

\begin{document}

	\newtheorem{thm}{Theorem}[section]
	\newtheorem{prop}[thm]{Proposition}
	\newtheorem{lem}[thm]{Lemma}
	\newtheorem{cor}[thm]{Corollary}
	\newtheorem{rem}[thm]{Remark}
	\newtheorem*{defn}{Definition}

	\newtheorem{definit}[thm]{Definition}
	\newtheorem{setting}{Setting}
	\renewcommand{\thesetting}{\Alph{setting}}
	
	\newcommand{\DD}{\mathbb{D}}
	\newcommand{\NN}{\mathbb{N}}
	\newcommand{\ZZ}{\mathbb{Z}}
	\newcommand{\QQ}{\mathbb{Q}}
	\newcommand{\RR}{\mathbb{R}}
	\newcommand{\CC}{\mathbb{C}}
	\renewcommand{\SS}{\mathbb{S}}

	\renewcommand{\theequation}{\arabic{section}.\arabic{equation}}

	\newcommand{\supp}{\mathop{\mathrm{supp}}}    
	
	\newcommand{\re}{\mathop{\mathrm{Re}}}   
	\newcommand{\im}{\mathop{\mathrm{Im}}}   
	\newcommand{\dist}{\mathop{\mathrm{dist}}}  
	\newcommand{\link}{\mathop{\circ\kern-.35em -}}
	\newcommand{\spn}{\mathop{\mathrm{span}}}   
	\newcommand{\ind}{\mathop{\mathrm{ind}}}   
	\newcommand{\rank}{\mathop{\mathrm{rank}}}   
	\newcommand{\Fix}{\mathop{\mathrm{Fix}}}   
	\newcommand{\codim}{\mathop{\mathrm{codim}}}   
	\newcommand{\conv}{\mathop{\mathrm{conv}}}   
	\newcommand{\epsi}{\mbox{$\varepsilon$}}
	\newcommand{\eps}{\mathchoice{\epsi}{\epsi}
		{\mbox{\scriptsize\epsi}}{\mbox{\tiny\epsi}}}
	\newcommand{\cl}{\overline}
	\newcommand{\pa}{\partial}
	\newcommand{\ve}{\varepsilon}
	\newcommand{\zi}{\zeta}
	\newcommand{\Si}{\Sigma}
	\newcommand{\cA}{{\mathcal A}}
	\newcommand{\cG}{{\mathcal G}}
	\newcommand{\cH}{{\mathcal H}}
	\newcommand{\cI}{{\mathcal I}}
	\newcommand{\cJ}{{\mathcal J}}
	\newcommand{\cK}{{\mathcal K}}
	\newcommand{\cL}{{\mathcal L}}
	\newcommand{\cN}{{\mathcal N}}
	\newcommand{\cR}{{\mathcal R}}
	\newcommand{\cS}{{\mathcal S}}
	\newcommand{\cT}{{\mathcal T}}
	\newcommand{\cU}{{\mathcal U}}
	\newcommand{\OM}{\Omega}
	\newcommand{\B}{\bullet}
	\newcommand{\ol}{\overline}
	\newcommand{\ul}{\underline}
	\newcommand{\vp}{\varphi}
	\newcommand{\AC}{\mathop{\mathrm{AC}}}   
	\newcommand{\Lip}{\mathop{\mathrm{Lip}}}   
	\newcommand{\es}{\mathop{\mathrm{esssup}}}   
	\newcommand{\les}{\mathop{\mathrm{les}}}   
	\newcommand{\nid}{\noindent}
	\newcommand{\pzr}{\phi^0_R}
	\newcommand{\pir}{\phi^\infty_R}
	\newcommand{\psr}{\phi^*_R}
	\newcommand{\pow}{\frac{N}{N-1}}
	\newcommand{\ncl}{\mathop{\mathrm{nc-lim}}}   
	\newcommand{\nvl}{\mathop{\mathrm{nv-lim}}}  
	\newcommand{\la}{\lambda}
	\newcommand{\La}{\Lambda}    
	\newcommand{\de}{\delta}    
	\newcommand{\fhi}{\varphi} 
	\newcommand{\ga}{\gamma}    
	\newcommand{\ka}{\kappa}   
	
	\newcommand{\core}{\heartsuit}
	\newcommand{\diam}{\mathrm{diam}}

	\newcommand{\lan}{\langle}
	\newcommand{\ran}{\rangle}
	\newcommand{\tr}{\mathop{\mathrm{tr}}}
	\newcommand{\diag}{\mathop{\mathrm{diag}}}
	\newcommand{\dv}{\mathop{\mathrm{div}}}
	
	\newcommand{\al}{\alpha}
	\newcommand{\be}{\beta}
	\newcommand{\Om}{\Omega}
	\newcommand{\na}{\nabla}
	
	\newcommand{\cC}{\mathcal{C}}
	\newcommand{\cM}{\mathcal{M}}
	\newcommand{\nr}{\Vert}
	\newcommand{\De}{\Delta}
	\newcommand{\cX}{\mathcal{X}}
	\newcommand{\cP}{\mathcal{P}}
	\newcommand{\om}{\omega}
	\newcommand{\si}{\sigma}
	\newcommand{\te}{\theta}
	\newcommand{\Ga}{\Gamma}
	
	\newcommand{\vV}{\mathbf{v}}
	\newcommand{\lbunu}{\ul{m}}
	\newcommand{\ca}{\tilde{a}}
	\newcommand{\Vve}{\ul{\varepsilon}}

	\title[Optimal quantitative stability for a Serrin-type problem in convex cones]{Optimal quantitative stability for a Serrin-type problem in convex cones}
	
	\author{Filomena Pacella}
	\address{F. Pacella. Dipartimento di Matematica, Sapienza Universit\`{a} di Roma, P.le Aldo Moro 2, 00185 Roma, Italy}
    \email{pacella@mat.uniroma1.it}
	
	\author{Giorgio Poggesi}
	\address{G. Poggesi. Department of Mathematics and Statistics, The University of Western Australia, 35 Stirling Highway, Crawley, Perth, WA 6009, Australia}
	%
	%
	\email{giorgio.poggesi@uwa.edu.au}
	
	\author{Alberto Roncoroni}
	\address{A. Roncoroni. Dipartimento di Matematica, Politecnico di Milano, Piazza Leonardo da Vinci 32, 20133, Milano, Italy}
    \email{alberto.roncoroni@polimi.it}

	
	\begin{abstract}
We consider a Serrin’s type problem in convex cones in the Euclidean space and motivated by recent rigidity results we study the quantitative stability issue for this problem. In particular,  we prove both sharp Lipschitz estimates for an $L^2-$pseudodistance and estimates in terms of the Hausdorff distance.
	\end{abstract}

\keywords{Serrin's overdetermined problem, convex cones, symmetry, rigidity, integral identities, stability, quantitative estimates}
\subjclass{Primary 35N25, 53A10, 35B35; Secondary 35A23}

\maketitle
	
	\section{Introduction}
The present paper deals with the quantitative stability of a rigidity result for a mixed boundary value Serrin-type problem in convex cones. Such a rigidity result was established in \cite{PT} (see also \cite{CR}) for convex cones that are smooth outside the origin. Here our analysis allows non-smooth cones such as cones with singularities possibly different from the origin alone. Similar results for (almost) constant mean curvature hypersurfaces have been recently obtained in \cite{Pog3}.
	
	Given a cone $\Si$ in $\RR^N$, $N\geq 2$, with vertex at the origin,  i.e. 
	$$
	\Sigma=\lbrace tx \, : \, x\in\omega\, , \, t\in(0,\infty)\rbrace\, , 
	$$
	where $\omega$ is an open connected set on the unit sphere $\mathbb{S}^{N-1}$; we consider a bounded domain (i.e., a bounded connected open set) $\Si \cap \Om$  -- where $\Om$ is a smooth bounded domain in $\RR^N$ --
	such that its boundary relative to the cone $\Ga_0:= \Si \cap \pa\Om$ is smooth,  while $\pa \Ga_0$ is a $N-2$-dimensional manifold and $\pa(\Si\cap\Om)\setminus \Ga_0$ is smooth enough outside a singular set $\mathcal{S} \subset \pa \Si$ of finite $\ell$-dimensional upper Minkowski content for some $0 \le \ell \le N-2$. To simplify matters, we also assume
	that $\cH^{N-2}(  \pa^*\Si  \cap  \pa \Ga_0) = \cH^{N-2}(\pa \Ga_0)$, where $\pa^* \Si$ denotes the smooth part of $\pa \Si$.  For further details on the setting, we refer to \cite[Setting A and Remark 2.2]{Pog3}. 
	We also set $\Ga_1 :=  \pa(\Si\cap\Om) \setminus \left( \ol{\Ga}_0 \cup \ol{\cS} \right)$ and denote with $\nu$ the (exterior) unit normal vector field to $\Ga_0 \cup \Ga_1$.
	We consider the following mixed boundary value problem:
	\begin{equation}\label{eq:problem torsion}
		\begin{cases}
			\De u = N \quad & \text{ in } \Si \cap \Om 
			\\
			u= 0 \quad & \text{ on } \Ga_0
			\\
			u_\nu=0 \quad & \text{ on } \Ga_1 .
		\end{cases}
	\end{equation}
	As in \cite{PT}, we assume that the solution $u$ of \eqref{eq:problem torsion} is of class 
\begin{equation}\label{reg_u}
W^{1,\infty}(\Si\cap\Om) \cap W^{ 2,2 }(\Si\cap\Om)\, ;
\end{equation}
such an assumption can be viewed as a gluing condition, and, as proved in \cite[Section 6]{PT} for cones smooth outside of the vertex, it is surely satisfied if $\ol{\Ga}_0$ and $\pa\Si$ intersect orthogonally.

As proved in \cite{Pog3} we have the following fundamental identity for Serrin's problem in $\Si$:
\begin{equation}\label{idwps}
		\int_{ \Si \cap \Om} (-u) \left\{ |\na ^2 u|^2- \frac{ (\De u)^2}{N} \right\} dx + \int_{\Ga_1} u \,  \langle \na^2 u \na u, \nu \rangle \, dS_x =
		\frac{1}{2}\,\int_{\Ga_0} \left( u_\nu^2- R^2\right) (u_\nu- \langle x-z , \nu \rangle )\,dS_x ,
	\end{equation}
	for every $z \in \RR^N$  such that
	\begin{equation}\label{eq:INTRO_inner product z in cone}
		\langle x-z ,\nu \rangle = 0 \quad \text{ for any } x \in \Ga_1 ,
	\end{equation}
where		$$R=\frac{N\,|  \Si\cap\Om |}{|\Ga_0|}.$$ 
If the cone $\Sigma$ is convex, such identity provides an alternative proof of \cite[Theorem 1.1]{PT}\footnote{In \cite{PT} two proofs were provided, the first following the tracks of \cite[Theorem 1]{BNSTARMA} and the second following the tracks of \cite[Theorem 1]{We}. The proof in \cite{Pog3} instead, follows the tracks of
\cite[Theorems I.1, I.2]{PS} and their subsequent development in \cite[Theorem 2.1]{MP2}.}.  Indeed, if the following overdetermined condition is in force:
\begin{equation}\label{overdet_con}
u_\nu=R \quad \text{on $\Gamma_0$}\, ,
\end{equation}
then \eqref{reg_u} is satisfied (this can be deduced from \cite{LS}) and \eqref{idwps} reads as
\begin{equation*}
		\int_{ \Si \cap \Om} (-u) \left\{ |\na ^2 u|^2- \frac{ (\De u)^2}{N} \right\} dx + \int_{\Ga_1} u \,  \langle \na^2 u \na u, \nu \rangle \, dS_x =0 \, .
	\end{equation*}
Moreover, being $u\leq 0$ in $\Sigma\cap\Omega$ (see e.g. \cite[Lemma 4.1]{Pog3}) and using the convexity of the cone one has (see e.g. \cite[Formula (3.9)]{PT})
\begin{equation}\label{convexity_PT}
 \int_{\Ga_1} u \,  \langle \na^2 u \na u, \nu \rangle \, dS_x\geq 0 \, ,
\end{equation}
and hence
\begin{equation*}
		\int_{ \Si \cap \Om} (-u) \left\{ |\na ^2 u|^2- \frac{ (\De u)^2}{N} \right\} dx  \leq 0 \, .
	\end{equation*}
But, on the other hand, from Cauchy-Schwarz inequality we easily have  
\begin{equation*}
		\int_{ \Si \cap \Om} (-u) \left\{ |\na ^2 u|^2- \frac{ (\De u)^2}{N} \right\} dx  \geq 0 \, .
	\end{equation*}
Hence,
\begin{equation*}
		\int_{ \Si \cap \Om} (-u) \left\{ |\na ^2 u|^2- \frac{ (\De u)^2}{N} \right\} dx = 0 \,,
	\end{equation*}
and so, being as $u<0$ in $\Si\cap\Om$ (see e.g. \cite[Lemma 4.2]{Pog3}), -- similarly to \cite{BNSTARMA,PT,PS,We} -- we deduce the following rigidity:
\begin{equation}
\Sigma\cap\Omega=\Sigma\cap B_R(z) \quad \text{ and } \quad u(x)=\frac{\left( \vert x- z\vert^2 -R^2\right)^2}{2}\, .
\end{equation}

	%
Identity \eqref{idwps} provides the starting point of our quantitative analysis. We refer the reader to \cite{ABR,BNST,CMV,Feldman, MP2, MP3, MP6, GO, O} and the survey \cite{CR_survey} for results related to the quantitative stability of the classical Serrin's problem (i.e., the particular case where $\Si=\RR^N$).  Roughly speaking what we want to prove is the following: if \eqref{overdet_con} is ``almost'' satisfied then the domain $\Om$ is ``close'' to the ball in a quantitative way.

We mention that,  the point $z\in\mathbb{R}^N$ can be characterized in terms of the
	linear space generated by the normal vector field $\nu (x)$ for $x \in \Ga_1$. In fact, being $\Si$ a cone with vertex at the origin we have that $\langle x, \nu \rangle = 0$ on $\Ga_1 \subset \pa \Si$, and hence \eqref{eq:INTRO_inner product z in cone} is equivalent to $\langle z, \nu \rangle = 0$ on $\Ga_1$. That is, $z \in \left[ \mathrm{span}\left\lbrace \nu(x)\, : \, x\in \Ga_1\right\rbrace \right]^{\bot}$, where $\left[ \mathrm{span}\left\lbrace \nu(x)\, : \, x\in \Ga_1\right\rbrace \right]^{\bot}$ is the orthogonal complement in $\RR^N$ of the vector subspace $\mathrm{span}\left\lbrace \nu(x)\, : \, x\in \Ga_1\right\rbrace\subseteq \RR^N$.
	In particular, 
	\begin{equation}\label{eq:intro k=N}
		\text{dim}\left( \mathrm{span}\left\lbrace \nu(x)\, : \, x\in \Ga_1\right\rbrace\right) = N
	\end{equation}
	is a sufficient condition that guarantees that $z$ must be the origin.
	Moreover,  condition \eqref{eq:intro k=N} is surely verified if $\Ga_1$ contains at least a transversally nondegenerate point, in the sense of the definition introduced in \cite{PT isoperimetric} (see also \cite[Proposition 2.15]{Pog3}).  
	In particular, this is always the case if $\Si$ is a strictly convex cone (and $\Ga_1 \neq \varnothing$). That a transversally nondegenerate point was sufficient to force $z$ to be the origin was noticed in \cite{PT isoperimetric}. The condition in \eqref{eq:intro k=N}, used here and in \cite{Pog3}, is more general and successfully applies to the study of the stability issue.
	To avoid excessive technicalities, the stability results are presented under the additional assumption that $\Ga_0$ and $\pa \Si$ intersect in a Lipschitz way so that $\Si\cap\Om$ is a Lipschitz domain.

	The tool
	that allows to fix the center $z$ of the approximating ball in $\left[ \mathrm{span}\left\lbrace \nu(x)\, : \, x\in \Ga_1\right\rbrace \right]^{\bot}$ (that is the origin whenever \eqref{eq:intro k=N} is in force) is the following weighted Poincar\'{e}-type inequality established in \cite{Pog4}:
	\begin{equation}\label{eq:INTRO new Poicare caso particolare p=2}
		\nr \vV \nr_{L^p (\Si \cap \Om)} \le \eta_{p,\al}( \Ga_1 , \Si\cap\Om)^{-1}  \, \nr \de_{\Ga_0}^\al  \na \vV \nr_{L^p (\Si \cap \Om)} , 
	\end{equation}
	which holds true for any $0 \le \al \le 1$ and $\vV : \Si\cap\Om \to \mathrm{span}\left\lbrace \nu(x)\, : \, x\in \Ga_1\right\rbrace \subseteq \RR^N$ such that $\vV \in W^{1,p}_\al (\Si\cap\Om)$ and $\langle \vV , \nu \rangle = 0$ a.e. in $\Ga_1$.
	
	Notice that, if we consider the function
	\begin{equation}\label{eq:INTRO def h}
		h := q - u , \quad \text{ where } q \text{ is the quadratic function defined as } q(x)=\frac12\, |x-z|^2 ,
	\end{equation}
	the choice $z = 0$ always guarantees that $\langle \na h , \nu \rangle = 0$ on $\Ga_1$, by the homogeneous Neumann condition $u_\nu=0$ on $\Ga_1$ and $\langle x,\nu \rangle =0$ on $\Ga_1\subset\pa\Si$; therefore, if $\mathrm{span}\left\lbrace \nu(x)\, : \, x\in \Ga_1\right\rbrace = \RR^N$ the new Poincar\'e-type inequality can be applied with $\vV:=\na h$.
	Such a weighted Poincar\'e-type inequality will allow us to deal with the weighted integral
	$$
	\int_{ \Si \cap \Om} (-u) \left\{ |\na ^2 u|^2- \frac{ (\De u)^2}{N} \right\} dx \, ,
	$$
	which appears in \eqref{idwps}. In the classical case $\Si=\RR^N$, different approaches to deal with such a weighted integral can be found in \cite{BNST, Feldman, MP2, MP3, MP6}.

	The quantitative stability results provided in the present paper include, as particular case\footnote{More precise general statements will be provided later on in this Introduction.} and when \eqref{eq:intro k=N} is in force, the following Lipschitz stability estimate for the $L^2$-pseudodistance of $\Si\cap\Om$ to $\Si \cap B_R (0)$:
	\begin{equation}\label{eq:intro k=N Serrin}
		\nr |x| - R \nr_{L^2(\Ga_0)} \le C \, \nr u_\nu^2 -R^2 \nr_{L^2(\Ga_0)} \, . 
	\end{equation}
The closeness in $L^2$-pseudodistance obtained here is stronger than the closeness in terms of the so called asymmetry in measure. In fact, clearly the $L^2$-pseudodistance is stronger than the $L^1$- pseudodistance, being as
$$
\nr |x| - R \nr_{L^1(\Ga_0)} \le |\Ga_0|^{1/2} \nr |x| - R \nr_{L^2(\Ga_0)} ,
$$
by H\"older's inequality.
In turn, \cite[Proposition 6.1]{CGPRS} informs us that the $L^1$-pseudodistance is stronger than the asymmetry in measure, that is
$$
| (\Si\cap\Om) \De (\Si\cap B_R (0))| \lesssim \nr |x| - R \nr_{L^1(\Ga_0)} .
$$

The constants in our quantitative estimates can be explicitly computed and estimated in terms of a few chosen geometrical parameters. At first, we obtain \eqref{eq:intro k=N Serrin} for an explicit constant $C$ only depending on $\eta_{2,1}(\Ga_1, \Si\cap\Om)$ and a lower bound $\ul{m}$ for $u_\nu$ on $\Ga_0$.

In \cite{Pog3}, new notions of uniform interior and exterior sphere conditions relative to the cone $\Si$ were introduced. These return the classical known uniform sphere conditions in the case $\Si =\RR^N$; when $\Si \subsetneq \RR^N$ they are related to how $\ol{\Ga}_0$ and $\pa\Si$ intersect (see \cite[Sections 4.1 and 4.2]{Pog3}), and both of them are always satisfied if $\ol{\Ga}_0$ and $\ol{\Ga}_1$ intersect orthogonally. As in the classical case $\Si=\RR^N$, these conditions revealed to be useful tools to perform barrier arguments in the mixed boundary value setting for $\Si \subset \RR^N$ to obtain uniform lower and upper bound for the gradient.
In particular,  if $\Sigma$ is convex, $\ul{r}_i$-uniform interior sphere condition relative to $\Si$ guarantees the validity of Hopf-type estimates: in fact, \cite[Lemma 4.4]{Pog3} gives that $u_\nu \ge \ul{r}_i$ on $\Ga_0$ so that we can take $\lbunu:=\ul{r}_i$. Hence, whenever $\Si\cap\Om$ satisfies the $\ul{r}_i$-uniform interior sphere condition relative to $\Si$, we obtain \eqref{eq:intro k=N Serrin} with an explicit $C=C(\eta_{2,1}(\Ga_1, \Si\cap\Om), \ul{r}_i)$.

Similarly to \cite{Pog3}, our method is robust enough to give a complete characterization of the stability issue even in absence of the assumption \eqref{eq:intro k=N}. In fact, in general we can set
\begin{equation}\label{def_k}
k:=\text{dim}\left( \mathrm{span}\left\lbrace \nu(x)\, : \, x\in \Ga_1\right\rbrace\right) ,
\end{equation}
which may be any integer $0 \le k \le N$, and obtain closeness of $\Si\cap \Om$ to $\Si\cap B_R (z)$ for some suitable point $z$ whose components in the $k$ directions spanned by $\mathrm{span}\left\lbrace \nu(x)\, : \, x\in \Ga_1\right\rbrace$ are set to be $0$.
General statements containing Lipschitz stability estimates for the $L^2$-pseudodistance are presented in what follows.

Notice that the case $\Ga_1 = \emptyset $ is included in our treatment (in that case, we have $k=0$).

Up to changing orthogonal coordinates, we can assume that $\mathrm{span}\left\lbrace \nu(x)\, : \, x\in \Ga_1\right\rbrace$ is the space generated by the first $k$ axes $\mathbf{e}_1, \dots, \mathbf{e}_k$.
Notice that, in this way, if we set $z \in \RR^N$ of the form
\begin{equation}\label{eq:INTRO_choice of z 1of2}
	z=(0, \dots , 0 , z_{k+1}, \dots, z_N) \in \RR^N ,
\end{equation}
it surely satisfies \eqref{eq:INTRO_inner product z in cone}.
We also fix
\begin{equation}\label{eq:INTRO_choice of z 2of2}
	z_i = \frac{1}{|\Si\cap\Om|} \int_{\Si\cap\Om} ( x_i - u_i(x) ) \, dx \quad \text{for } i=k+1, \dots, N ,
\end{equation}
where $u_i$ denotes the $i$-th partial derivative of $u$ and $x_i$ the $i$-th component of the vector $x \in \RR^N$.
With this choice of $z$, if we consider the harmonic function $h$ defined in \eqref{eq:INTRO def h},
we have that
\begin{equation}\label{eq:gradienteproiettatosuspan nu}
	(h_1, \dots, h_k , 0 , \dots, 0) \in \mathrm{span}\left\lbrace \nu(x)\, : \, x\in \Ga_1\right\rbrace \subseteq \RR^N ,
	\quad \quad
	\langle (h_1, \dots, h_k , 0 , \dots, 0), \nu \rangle = \langle \na h , \nu \rangle  = 0 \, \text{ on } \Ga_1 
\end{equation}
and
\begin{equation}\label{eq:gradienteparte ortogonale a span nu}
	\int_{\Si \cap \Om} h_i \, dx = 0 \quad \text{for } i = k+1, \dots, N .
\end{equation}
The identity $\langle \na h , \nu \rangle = 0$ on $\Ga_1$ easily follows by \eqref{eq:INTRO_inner product z in cone} and the
%
%
Neumann condition $u_\nu = 0$ on $\Ga_1$.

This will allow to use the Poincar\'e inequality \eqref{eq:INTRO new Poicare caso particolare p=2} with $\vV := (h_1, \dots, h_k , 0 , \dots, 0)$ and the (classical) weighted Poincar\'e inequality for functions with zero mean
\begin{multline}\label{Poincare_intro}
	\nr v \nr_{L^p(\Si\cap\Om)} \le \mu_{p,\al} (\Si\cap\Om)^{-1} \nr \de_{\pa (\Si\cap\Om)}^\al \,  \na v  \nr_{L^p(\Si\cap\Om)}, 
	\,
	\text{ for } 0\le\al\le1 , \, \,  \\ v \in L^p (G) \cap W^{1,p}_{loc} (G)
	\text{ with }
	v_{\Si\cap\Om}=0,
\end{multline}
with $v:=h_i$ for $i=k+1, \dots, N$.

Setting
\begin{equation}\label{def:La p al k}
	\La_{p, \al}(k) :=
	\begin{cases}
		\mu_{p, \al}(\Si\cap\Om)^{-1} \quad & \text{if } k=0
		\\
		\eta_{p, \al } (\Ga_1 , \Si \cap \Om )^{-1} \quad & \text{if } k=N
		\\
		\max\left[ \mu_{p, \al}(\Si\cap\Om)^{-1} , \, \eta_{p, \al} ( \Ga_1 ,  \Si \cap \Om )^{-1}  \right] \quad & \text{if } 1 \le k \le N-1 ,
	\end{cases}
\end{equation}
where $\mu_{p,\al}(\Si\cap\Om)$ and $\eta_{p, \al} ( \Ga_1 , \Si \cap \Om )$ are the constants in \eqref{Poincare_intro} and \eqref{eq:INTRO new Poicare caso particolare p=2} (see also Lemma \ref{lem:BoasStraube} and Theorem \ref{thm:Poincare new in general} below).  We are now ready to present the sharp stability result for the $L^2$-pseudodistance.
	
	\begin{thm}[Lipschitz stability for Serrin's problem in terms of an $L^2$-psudodistance]
		\label{thm:INTRO_Serrinstab Lipschitz}
		Let $\Sigma$ be a convex cone and let $\Si\cap\Om$ be as
		described above. Let $u$ be a solution of \eqref{eq:problem torsion} satisfying \eqref{reg_u} and such that $u_\nu\geq \lbunu$ on $\Ga_0$, for some $\lbunu > 0$. Given the point $z$ defined in \eqref{eq:INTRO_choice of z 1of2} and \eqref{eq:INTRO_choice of z 2of2}, we have that
		\begin{equation}\label{eq:INTRO_Serrin}
			\nr |x-z| - R \nr_{L^2(\Ga_0)} \le C \, \nr u_\nu^2 -R^2 \nr_{L^2(\Ga_0)} ,
		\end{equation}
		where the positive constant $C$ can be explicitly estimated as follows
		\begin{equation*}
			C \le \frac{1}{2 \, \lbunu} \left( 2 N \, \La_{2,1}(k)^2  + 3 \right) .
		\end{equation*}
		%
		%
Whenever $\Si\cap\Om$ satisfies the $\ul{r}_i$-uniform interior sphere condition relative to $\Si$, we can take $\lbunu:= \ul{r}_i$.
	\end{thm}
	
Regarding the last sentence of the theorem we refer to Remark \ref{remark_palla int}.	

Our second quantitative stability result is presented in the following theorem. Before stating it, we need to introduce the following notations. Given the point $z\in\mathbb{R}^N$ chosen in \eqref{eq:INTRO_choice of z 1of2} and in \eqref{eq:INTRO_choice of z 2of2} we define 
\begin{equation}\label{def_rho}
\rho_e=\max_{x\in\overline{\Gamma}_0}\vert x-z\vert \quad \text{ and } \quad \rho_i=\min_{x\in\overline{\Gamma}_0}\vert x-z\vert \, ,
\end{equation}
so that we have 
$$
\Gamma_0\subseteq \left(\overline{B}_{\rho_e}(z) \setminus B_{\rho_i}(z)\right)\cap\Sigma\, . 
$$

Given $\theta \in \left( 0, \pi /2 \right]$ and $\ca >0$, we say that a set $G$ satisfies the \textit{$(\theta, \ca )$-uniform interior cone condition}, if for every $x \in \pa G$ there is a unit vector $\om=\om_x$ such that the cone with vertex at the origin, axis $\om$, opening width $\te$, and height $\ca$ defined by
$$
\cC_{\om}=\left\{ y \, : \, \langle y , \om \rangle > |y| \cos(\theta) , \, |y|< \ca \right\}
$$
is such that
\begin{equation*}
	w + \cC_{\om}  \subset G \ \text{ for every } \ w \in B_{\ca} (x) \cap \ol{G} .
\end{equation*}
Such a condition is equivalent to Lipschitz-regularity of the domain; more precisely, it is equivalent to the strong local Lipschitz property of Adams \cite[Pag 66]{Adams}.

The landmark result of the next theorem is to estimate the difference $\rho_e-\rho_i$ in terms of the $L^2-$norm of the function $u_\nu - R$. Explicitly we have the following

	\begin{thm}[Stability in terms of $\rho_e - \rho_i$ for Serrin's problem in cones] 
		\label{thm:Serrin general stability rhoe rhoi}
				Let $\Sigma$ be a convex cone and let $\Si\cap\Om$ be as described above. Let $u$ be a solution of \eqref{eq:problem torsion} satisfying \eqref{reg_u} and assume that $\Si\cap\Om$ satisfies
		the $(\te,\ca)$-uniform interior cone condition. 
		\par
		Let $z \in \RR^N$ be the point chosen in \eqref{eq:INTRO_choice of z 1of2}-\eqref{eq:INTRO_choice of z 2of2}.
		Then, we have that
		\begin{equation}
			\label{eq:stability Serrin rhoei general}
			\rho_e - \rho_i  \le 
			C \,
			\begin{cases}
				\nr u_\nu - R \nr_{L^2(\Ga_0)} \max \left[ \log \left( \frac{1  }{ \nr u_\nu - R \nr_{L^2(\Ga_0)} } \right) , 1 \right],   \ &\mbox{if } N=2 ,
				\\
				\nr u_\nu - R \nr^{\frac{2}{N}}_{L^2(\Ga_0)},  \ &\mbox{if } N \ge 3 .
			\end{cases}
		\end{equation}
		The constant $C$ can be explicitly estimated only in terms of $N, \ca , \te$, the constant $\eta_{2, 1}(\Ga_1,\Si\cap\Om)$ from Theorem \ref{thm:Poincare new in general}, the diameter $d_{\Si\cap\Om}$, $\lbunu$ defined in \eqref{def:lower bound unu}, and $\nr \na u \nr_{L^{\infty}(\Si\cap\Om)}$.
		
		Whenever $\Si\cap\Om$ satisfies the $\ul{r}_i$-uniform interior sphere condition relative to $\Si$, we can take $\lbunu:= \ul{r}_i$.
		If $\Si\cap\Om$ satisfies the $\ul{r}_e$-uniform exterior sphere condition relative to $\Si$, $\nr \na u \nr_{L^{\infty}(\Si\cap\Om)}$ can be explicitly estimated in terms of $N$, $d_{\Si\cap\Om}$ and $\ul{r}_e$.
	\end{thm}

Regarding the last two sentences of the theorem, we refer to Remarks \ref{remark_palla int} and \ref{rem:stima norma na u infinito SicapOm con N r_e diam}. 
In addition to the previous theorem, in Theorem \ref{thm:Improved Serrin stability rhoe rhoi} we also show that, under certain geometrical assumptions (i.e., Definition \ref{def:interior sphere relative to cone} and \eqref{eq:condition relation dist improved}), the stability profile in \eqref{eq:stability Serrin rhoei general} can be improved. As noticed in Remark \ref{rem:NEW on ADDITIONAL ASSUMPTION} such additional assumptions are automatically satisfied when $\ol{\Ga}_0$ and $\ol{\Ga}_1$ intersect orthogonally, and they are also trivially satisfied when $\pa{\Ga}_0=\varnothing$.

We point out that
different choices of the point $z$ can lead to alternative stability results. For instance, we may avoid using \eqref{eq:INTRO new Poicare caso particolare p=2} and hence completely remove the dependence on $\eta_{p, \al}( \Ga_1 ,  \Si \cap \Om )^{-1}$ for any $0 \le k \le N$, at the cost of leaving the point $z$ free to have non-zero components also in the directions
spanned by $\nu$ on $\Ga_1$: we refer to Subsection \ref{subsec:alternative choices z} for details.
	

	When $\Si=\RR^N$, Theorem \ref{thm:INTRO_Serrinstab Lipschitz} returns a sharp stability result for the classical Serrin's problem in the spirit of \cite{Feldman}, whereas Theorem \ref{thm:Improved Serrin stability rhoe rhoi} returns variants of the stability estimates established in \cite{MP3} (for $N \neq 3$) and \cite{MP6} (for $N=3$). We refer to Subsection \ref{subsec:classical case Si=RN} and Theorem \ref{thm:classical INTRO in R^N} for details.
	
%
%
%
%
%

We conclude the introduction by mentioning that the hypothesis that the cone is convex is motivated by the rigidity result of \cite{PT} which uses the inequality \eqref{convexity_PT} which, in turn, relies on the convexity of the cone.  The fact that the convexity of the cone plays a role to get the rigidity theorems is clear by analogous results for constant mean curvature surfaces and for the isoperimetric problem (see e.g.  \cite{CROS,FI,LP,RR}); the convexity has also been very important to prove Liouville-type results for the critical $p-$Laplace equation (see \cite{CFR,LPT}) and rigidity results such as radial symmetry à la Gidas-Ni-Nirenberg (see \cite{DPV2} and also the recent paper \cite{CPP} where it is shown that the result does not hold in general nonconvex cones).

Let us observe that in \cite{BF} the isoperimetric inequality is also obtained for almost convex cones; we believe that, similarly, the rigidity result of \cite{PT} should hold for almost convex cones and, consequently, our quantitative estimates should be extended to this case.  However, it is important to remark that rigidity results, both from overdetermined torsion problem and for the soap bubble one, cannot be obtained in general non-convex cones as shown in \cite{IPT}.

\subsection*{Organization of the paper} The paper is organized as follows. In Section \ref{sec:Poincare} we collect some preliminary estimates, in particular the Poincar\'e-type inequalities that are useful to obtain our stability results and Lipschitz growth estimates for $u$ from $\Ga_0$. Section \ref{sec:sharp stability} contains the stability analysis  in terms of the $L^2$-pseudodistance, including the proof of Theorem \ref{thm:INTRO_Serrinstab Lipschitz}. Section \ref{sec:stability rhoe-rhoi} provides the stability results in terms of $\rho_e-\rho_i$ and contains the proof of Theorem \ref{thm:Serrin general stability rhoe rhoi} and its improved version given in Theorem \ref{thm:Improved Serrin stability rhoe rhoi}. Finally,  in Section \ref{sec:additional remarks} we discuss the corresponding stability results for alternative choices of the point $z$ and the classical case $\Si=\RR^N$.

\section{Preliminary estimates}\label{sec:Poincare}
	
	In this section we collect some preliminary estimates that we are going to use in the sequel. In particular, we start recalling some weighted Poincar\'{e}-type inequalities and then we prove some Lipschitz growth estimates for the function $u$ from the boundary $\Gamma_0$.
	
	In what follows, for a set $G \subset \RR^N$ and a function $v: G \to \RR$, $v_G$ denotes the {\it mean value of $v$ in $G$}, that is
	$$
	v_G= \frac{1}{|G|} \, \int_G v \, dx.
	$$
	Also, denoting with $\de_{\pa G} (x)$ the distance of a point $x$ in $G$ to the boundary $\pa G$, for a function $v:G \to \RR$ we define
	$$
	\nr \de_{\pa G}^\al \, \na v \nr_{L^p (G)} = \left( \sum_{i=1}^N \nr \de_{\pa G}^\al \,  v_i \nr_{L^p (G)}^p \right)^\frac{1}{p} \quad \mbox{and} \quad
	\nr \de_{\pa G}^\al \, \na^2 v \nr_{L^p (G)} = \left( \sum_{i,j=1}^N \nr \de_{\pa G}^\al \, v_{ij} \nr_{L^p (G)}^p \right)^\frac{1}{p},
	$$
	for $0 \le \al \le 1$ and $p \in [1, \infty)$.
	
	We first recall the following \emph{weighted Poincar\'{e}-type inequality} which can be found in \cite{BS}.
	
	\begin{lem}\label{lem:BoasStraube}
		Let $G \subset\RR^N$, $N\ge 2$, be a bounded domain with boundary $\pa G$ of class $C^{0,\al}$,
		$0 \le \al \le 1$ and consider $p \in \left[ 1, \infty \right)$. Then, there exists a positive constant, $\mu_{p, \al} (G)$ such that
		\begin{equation}
			\label{eq:BoasStraube-poincare}
			\nr v - v_G \nr_{L^p(G)} \le \mu_{p, \al} (G)^{-1} \nr \de_{\pa G}^{\al} \, \na v  \nr_{L^p(G)},
		\end{equation}
		for every function $v \in L^p (G) \cap W^{1,p}_{loc} (G)$.
		\par
		In particular, if $G$ has a Lipschitz boundary, the number $\al$ can be replaced by any exponent in $[0,1]$. 
	\end{lem}
	
	\begin{rem}
		{\rm
			When $\al=0$ we understand the boundary of $G$ to be locally the graph of a continuous function. 
		}
	\end{rem}

Secondly, we recall that in \cite{HS} inequality \eqref{eq:BoasStraube-poincare} has been strengthened, provided $p(1 -\al)<N$.  We report here a reformulation of the result in \cite{HS} and we refer to \cite[Lemma 2.1]{MP3} for a proof.

\begin{lem}\label{lem:Hurri}
	Let $G \subset \RR^N$ be a bounded $b_0$-John domain, and consider three numbers $r, p, \al$ such that 
	\begin{equation}\label{eq:r p al in Hurri}
		1 \le p \le r \le \frac{Np}{N-p(1 - \al )} , \quad p(1 - \al)<N , 
		\quad 0 \le \al \le 1 .
	\end{equation}
	Then, there exists a positive constant $\mu_{r, p, \al} (G)$ such that
	\begin{equation}
		\label{eq:John-Hurri-poincare}
		\nr v - v_G \nr_{L^r (G)} \le \mu_{r, p, \al} (G)^{-1} \nr \de_{\pa G}^{\al} \, \na v  \nr_{L^p(G)},
	\end{equation}
	for every function $v\in L^1_{loc}(G)$ such that $\de_{\pa G}^{\al} \, \na v \in L^p (G)$ .
\end{lem}

	The class of \emph{John domain} is huge: it contains Lipschitz domains, but also very irregular domains with fractal boundaries as, e.g., the Koch snowflake.
	Roughly speaking, a domain is a $b_0$-John domain if it is possible to travel from one point of the domain to another without going too close to the boundary.
	The formal definition is the following: a domain $G$ in $\RR^N$ is a {\it $b_0$-John domain}, $b_0 \ge 1$, if each pair of distinct points $a$ and $b$ in $G$ can be joined by a curve $\ga: \left[0,1 \right] \rightarrow G$ such that
	%
	%
	\begin{equation*}
		\de_{ \pa G} (\ga(t)) \ge b_0^{-1} \min{ \left\lbrace |\ga(t) - a|, |\ga(t) - b| \right\rbrace  }.
	\end{equation*}
	The notion could be also defined through the so-called {\it $b_0$-cigar} property (see \cite{Va}).
	
	\begin{rem}[Explicit estimates of the constants and geometric dependence]\label{rem:stime mu HS}
		{\rm
			The best constant is characterized by the (solvable) variational problem 
			\begin{equation*}
				\mu_{r, p, \al} (G) = 
				\min \left\{ \nr \de_{ \pa G}^{\al} \, \na v  \nr_{L^p(G)} : \nr v \nr_{L^r(G)} = 1 \text{ in } G,  v_{G} = 0 \right\}.
			\end{equation*} 
		
			Explicit estimates are provided by \cite[Remark 2.4]{MP3}, exploiting the fact that the proofs in \cite{HS1, HS} have the benefit of giving an explicit upper bound for the Poincar\'e constants.
			
			(i) For $\mu_{r, p, \al} (G)^{-1}$, we have that
			\begin{equation*}
				\mu_{r, p, \al} (G)^{-1} \le k_{N,\, r, \, p,\, \al} \, b_0^N |G|^{\frac{1-\al}{N} +\frac{1}{r} +\frac{1}{p} } .
			\end{equation*}
			
			(ii)  In the sequel we will also need an explicit estimate for the constant $\mu_{p,0} (G)$ appearing in \eqref{eq:BoasStraube-poincare} in the case $\al=0$.
			By putting together \cite[Theorem 8.5]{HS1} and \cite[Theorem 8.5]{MarS}, \cite[item (ii) of Remark 2.4]{MP3} informs that
			\begin{equation*}
				\mu_{p,0} (G)^{-1} \le k_{N, \, p} \, b_0^{3N(1 + \frac{N}{p})} \, d_G.
			\end{equation*}
			
			(iii) 	If $G$ satisfies the $(\te, \ca)$-uniform interior cone condition (defined in the Introduction), then it is a $b_0$-John domain and $b_0$ can be explicitly estimated in terms of $\te$, $\ca$, and $d_G$: see \cite[Lemma A.2]{MP7}.
		}
	\end{rem}

	
We now turn our attention to \emph{weighted Poincar\'{e}-type inequalities for vector fields,} in particular the next two theorems can be found in \cite{Pog4}.
	\begin{thm}\label{thm:Poincare new in general}
		Let $\Si$ be a cone and let $\Om$ a smooth bounded domain in $\RR^N$. Given $1 \le p < +\infty$ and $0 \le \al \le 1$, let $\Si\cap\Om \subset \RR^N$ be a bounded Lipschitz
		domain. Then, there exists a positive constant $ \eta_{p, \al}(\Ga_1 , \Si\cap\Om)$ (depending on $N$, $p$, $\al$, $\Ga_1$ and $\Si\cap\Om$) such that
		\begin{equation}\label{eq:Poincare new RN}
			\nr \vV \nr_{L^p(\Si\cap\Om)} \le \eta_{p, \al}(\Ga_1 , \Si\cap\Om)^{-1} \, \nr \de_{ \Ga_0}^\al D \vV \nr_{L^p(\Si\cap\Om)} ,
		\end{equation}
		for every $\vV: \Si\cap\Om \to \mathrm{span} \lbrace \nu (x) \, : \, x \in \Ga_1\rbrace \subseteq \RR^N$ 
		belonging to $W^{1,p}_\al (\Si\cap\Om)$ and such that $\langle \vV ,\nu  \rangle =0 $ a.e. on $\Ga_1$.
		Here and in the following, $W^{1,p}_\al (\Si\cap\Om)$ denotes the weighted Sobolev space with norm given by $\nr \vV \nr_{L^p(\Si\cap\Om)} + \nr \de_{\Ga_0}^\al D \vV \nr_{L^p(\Si\cap\Om)}$.
	%
	\end{thm}

\begin{rem}
	{\rm
		If \eqref{eq:intro k=N} is in force, 
		then \eqref{eq:Poincare new RN} holds true for any vector field $\vV: \Si\cap\Om \to \RR^N$ belonging to $W^{1,p}_\al (\Si\cap\Om)$ such that $\langle \vV , \nu \rangle = 0$ a.e. on $\Ga_1$. The existence of a transversally nondegenerate point on $\Ga_1$ (in the sense of the definition introduced in \cite{PT isoperimetric}) is sufficient for the validity of \eqref{eq:intro k=N}. In particular, this is always
		the case if $\Si$ is a strictly convex cone (and $\Ga_1 \neq \varnothing$).
	}
\end{rem}
	
As before,  in the case $p( 1 - \al) < N$, we have the following strengthened version of \eqref{eq:Poincare new RN}.
	\begin{thm}\label{thm:Strengthened Poincare new RN}
		Let $\Si\cap\Om \subset \RR^N$ be a bounded Lipschitz domain as before.
		%
		Let $r, p, \al$ be three numbers satisfying \eqref{eq:r p al in Hurri}. If $\langle \vV , \nu \rangle = 0$ a.e. in $\Ga_1$, then there exists a positive constant $ \eta_{r,p,\al}(\Ga_1, \Si\cap\Om)$ (depending on $N$, $r$, $p$, $\al$, $\Ga_1$ and $\Si\cap\Om$) such that
		\begin{equation}\label{eq:Strengthened Poincare new RN}
			\nr \vV \nr_{L^{r}(\Si\cap\Om)} \le \eta_{r,p,\al}(\Ga_1, \Si\cap\Om)^{-1} \, \nr \de_{ \Ga_0}^\al D \vV \nr_{L^p(\Si\cap\Om)} ,
		\end{equation}
		for for every $\vV: \Si\cap\Om \to \mathrm{span} \lbrace \nu (x) \, : \, x \in \Ga_1\rbrace \subseteq \RR^N$ belonging to $W^{1,p}_\al (\Si\cap\Om)$.
		Moreover, we have that
		\begin{equation*}
			\eta_{r,p,\al}(\Ga_1,\Si\cap\Om)^{-1} 
			\le 
			\max\left\lbrace 
			| \Si\cap\Om|^{ \frac{1}{r} - \frac{1}{p} } \, \, \eta_{p,\al}(\Ga_1, \Si\cap\Om )^{-1} 
			, \,
			\mu_{r, p, \al} ( \Si\cap\Om )^{-1} 
			\right\rbrace ,
		\end{equation*}
		where $\mu_{r, p, \al} (\Si\cap\Om)^{-1}$ and $\eta_{p,\al}(\Ga_1, \Si\cap\Om)^{-1}$ are those appearing in Lemma \ref{lem:Hurri} and Theorem \ref{thm:Poincare new in general}.
	\end{thm}

		

We now prove some Lipschitz growth estimates for the function $u$ from the boundary $\Gamma_0$. We firstly recall the following definition, which was introduced in \cite{Pog3}.

\begin{definit}\label{def:interior sphere relative to cone}
	We say that $\Si\cap\Om$ satisfies the $\ul{r}_i$-uniform interior sphere condition relative to the cone $\Si$, if for each $x \in \ol{\Ga}_0$ there exists a touching ball
	%
	%
	%
	of radius $\ul{r}_i$ such that
	
	(i) its center $x_0$ is contained in $\ol{\Si \cap \Om}$
	
	and
	
	(ii) its closure intersects $\ol{\Ga}_0$ only at $x$.
\end{definit}

We secondly recall the following Lemma, proved in \cite[Lemma 4.2]{Pog3}

\begin{lem}
	\label{lem:relation u dist general}
	Let $\Si$ be a cone and let $u$ be the solution of \eqref{eq:problem torsion}.
	We have that
	\begin{equation}\label{eq:relation u dist general}
		-u(x)\ge\frac12\,\de_{\pa(\Si\cap\Om)}(x)^2 \ \mbox{ for every } \ x\in\ol{\Si\cap\Om} ,
	\end{equation}
	where $\de_{\pa(\Si\cap\Om)} (x)$ denotes the distance of $x$ to $\pa(\Si\cap\Om)$.
	
	If $\Si$ is a convex cone, then we have that
	\begin{equation}\label{eq:non serve ma serve inproof reldist finer}
		-u(x) \ge \frac{1}{2}\,\de_{\Ga_0} (x)^2  \ \mbox{ for every } \ x\in\ol{\Si\cap\Om} ,
	\end{equation}	
	where $\de_{\Ga_0} (x)$ denotes the distance of $x$ to $\Ga_0$.
\end{lem}

We are now ready to prove the following finer version of the previous Lemma, under suitable additional assumptions.

\begin{lem}\label{lem:realtion dist IMPROVED}
	Let $u$ be the solution of \eqref{eq:problem torsion}. If $\Si$ is convex and $\Si\cap\Om$ satisfy the $\ul{r}_i$-uniform interior sphere condition with radius $\ul{r}_i$, and
	\begin{equation}\label{eq:condition relation dist improved}
		\begin{split}
			& \text{for any $x \in \ol{\Si\cap\Om}$ such that its closest point $\ul{x}$ to $\ol{\Ga}_0$ belongs to $\pa \Ga_0$,}
			\\
			& \text{the ball $B_{\ul{r}_i} \left( \ul{x} + \ul{r}_i \frac{x- \ul{x}}{|x- \ul{x} |} \right) $ is a touching ball at $\ul{x}$ relative to $\Si$ (as in Definition \ref{def:interior sphere relative to cone}), 
			}
		\end{split}
	\end{equation}
	then we have that
	\begin{equation}
		\label{eq:relation u dist improved}
		-u(x) \ge \frac{\ul{r}_i}{2}\,\de_{\Ga_0} (x)  \ \mbox{ for every } \ x\in\ol{\Si\cap\Om}.
	\end{equation}
\end{lem}
\begin{proof}
	By \eqref{eq:non serve ma serve inproof reldist finer}, \eqref{eq:relation u dist improved} certainly holds if $\de_{\Ga_0}(x) \ge \ul{r}_i$. If $\de_{\Ga_0 } (x) < \ul{r}_i$, instead, let $\ul{x}$ be the closest point in $\ol{\Ga}_0$ to $x$ and call $B:=B_{\ul{r}_i} \left( \ul{x} + \ul{r}_i \frac{x- \ul{x} }{|x- \ul{x} |} \right)$ the touching ball at $\ul{x} \in \ol{\Ga}_0$ which contains $x$. The existence of such a ball is guaranteed by Definition \ref{def:interior sphere relative to cone} and either \eqref{eq:condition relation dist improved} (if $\ul{x} \in \pa{\Ga}_0$) or the fact that $\Om$ is $C^1$ (if $\ul{x} \in \Ga_0$). By Definition \ref{def:interior sphere relative to cone}, the center $\ul{x}_0:= \ul{x} + \ul{r}_i \frac{x- \ul{x} }{|x- \ul{x} |} $ of the touching ball belongs to $\ol{\Si}\cap\Om$. Setting $w(y)=\left(|y - \ul{x}_0|^2 - \ul{r}_i^2 \right)/2$, we get that
	\begin{equation}\label{eq:system for comparison Hopf}
		\begin{cases}
			\De (w-u) = 0 \quad & \text{ in } \Si \cap B
			\\
			w-u \ge 0 \quad & \text{ on } \Si \cap \pa B
			\\
			w_\nu - u_\nu \ge 0 \quad & \text{ on } \pa \Si \cap B.
		\end{cases}
	\end{equation}
The last boundary condition holds being as $\Si\cap B$ star-shaped with respect to $\left( x_0 + \ul{r}_i \frac{x-x_0}{|x-x_0|} \right) \in \ol{\Si} \cap B$.
	By comparison (\cite[Lemma 4.1]{Pog3} with $f:=w-u$) we have that $w \ge u$ in $\Si\cap B$, and hence, being as $x \in \Si\cap B$,
	$$
	-u(x) \ge \frac12\,(|x - \ul{x}_0|^2- \ul{r}_i^2)=
	\frac12\,( \ul{r}_i + |x- \ul{x}_0| )(\ul{r}_i -|x- \ul{x}_0|)\ge\frac12\,\ul{r}_i \,(\ul{r}_i -|x- \ul{x}_0|).
	$$
	This implies \eqref{eq:relation u dist improved}, since $\ul{r}_i - |x- \ul{x}_0|=\de_{\Ga_0 } (x)$.
\end{proof}

\begin{rem}\label{rem:NEW on ADDITIONAL ASSUMPTION}
{\rm
As noticed in \cite[Section 4.1]{Pog3}, Definition \ref{def:interior sphere relative to cone}, returns the classical uniform interior sphere condition\footnote{Since, in general, we assume $\Om$ to be smooth (say, at least, $C^2$), it surely satisfies the classical uniform sphere conditions: see also Remark \ref{rem:final remark on regularity in the classical setting}.} in the case $\Si =\RR^N$, whereas when $\Si \subsetneq \RR^N$ it is related to how $\ol{\Ga}_0$ and $\pa\Si$ intersect; in fact, it is surely satisfied if $\ol{\Ga}_0$ and $\ol{\Ga}_1$ intersect orthogonally. We point out that also the additional assumption in \eqref{eq:condition relation dist improved} is automatically satisfied whenever $\ol{\Ga}_0$ and $\ol{\Ga}_1$ intersect orthogonally, and it is trivially satisfied whenever $\pa\Ga_0 = \varnothing$: in the last case, Lemma \ref{lem:realtion dist IMPROVED} reduces to \cite[(3.4)]{MP2}. 
}
\end{rem}

	\section{Sharp quantitative stability in terms of an \texorpdfstring{$L^2$}{L2}-pseudodistance: proof of Theorem~\ref{thm:INTRO_Serrinstab Lipschitz}}\label{sec:sharp stability}
	
	From now on, we consider $\Si$ and $\Om$ as in the setting described at the beginning of the Introduction, and in addition we assume the cone $\Si$ to be convex and that $\Si$ and $\Om$ intersect in a Lipschitz way so that $\Si\cap\Om$ is a Lipschitz domain.

	We set $k$ as in \eqref{def_k} and $z\in\mathbb{R}^N$ of the form \eqref{eq:INTRO_choice of z 1of2} such that \eqref{eq:INTRO_choice of z 2of2} holds. As already observed in the Introduction, with this choice of $z$, if we consider the harmonic function $h$ defined in \eqref{eq:INTRO def h} we have that \eqref{eq:gradienteproiettatosuspan nu} and \eqref{eq:gradienteparte ortogonale a span nu} hold true. Moreover, by direct computation, it is easy to check that $|\na^2 h|^2$ equals the Cauchy-Schwarz deficit for $\na^2 u$, that is,
		\begin{equation}\label{eq:hessiana h}
			|\na^2 h|^2 = | \na^2 u|^2- \frac{(\De u)^2}{N}	\quad \text{ in } \, \Si\cap\Om .
		\end{equation}
With the previous notations, we can now establish the following.
	\begin{lem}\label{lem:Mixed pp Poincareaigradienti}
		For $0 \le \al \le 1$ and $1 \le p < \infty$, we have that
		\begin{equation*}
			\nr \na h \nr_{L^p (\Si\cap\Om)} \le C \, \nr \de_{\Ga_0}^{\al} \, \na^2 h  \nr_{L^p(\Si\cap\Om)} ,
		\end{equation*}
		for some positive constant $C$ satisfying $C \le \La_{p,\al}(k)$, where the constant $\La_{p,\al}(k)$ is defined in \eqref{def:La p al k}.
	\end{lem}
	\begin{proof}
		%
		%
		In light of \eqref{eq:gradienteproiettatosuspan nu}, we can apply \eqref{eq:Poincare new RN} in Theorem \ref{thm:Poincare new in general} 
		with 
		$\vV:= (h_1, \dots, h_k , 0 , \dots, 0) $
		to get that
		\begin{equation}\label{eq:dimprovaPoincaregradmixed}
			\left( \sum_{i=1}^k \nr h_i \nr_{L^p (\Si\cap\Om)}^p \right)^{1/p}\le \eta_{p, \al}(\Ga_1 , \Si \cap \Om )^{-1} \, \left( \sum_{i,j=1}^k \nr \de_{\Ga_0}^{\al} \, h_{ij} \nr_{L^p (\Si\cap\Om)}^p \right)^{1/p}.
		\end{equation}
		In light of \eqref{eq:gradienteparte ortogonale a span nu}, we can apply \eqref{eq:BoasStraube-poincare} to each first partial derivative $h_i$ of $h$, $i=k+1, \dots, N$; notice that in those applications of \eqref{eq:BoasStraube-poincare} we can replace $\de_{\pa(\Si\cap\Om)}$ with $\de_{\Ga_0}$, being as $ \de_{\pa(\Si\cap\Om)}(x) \le \de_{\Ga_0}(x)$. 
		
		Raising to the power of $p$ those inequalities and \eqref{eq:dimprovaPoincaregradmixed}, and then summing up, the conclusion easily follows.
	\end{proof}

%

In what follows we show that we can obtain explicit ad hoc trace-type inequalities for $h$ and $\na h$ whenever we have at our disposal
	a positive lower bound $\lbunu$ for $|\na u|$ on $\Ga_0$, i.e.,
	\begin{equation}\label{def:lower bound unu}
		u_\nu \ge \lbunu > 0 \quad \text{on } \Ga_0 .
	\end{equation}
	
\begin{rem}\label{remark_palla int}
{\rm
	
We mention that a geometric condition that guarantees the validity of \eqref{def:lower bound unu} is the uniform interior sphere condition relative to $\Si$ of Definition \ref{def:interior sphere relative to cone}; indeed if $\Si\cap\Om$ satisfies the $\ul{r}_i$-uniform interior sphere condition relative to $\Si$, then \cite[Lemma 4.4]{Pog3} ensures that \eqref{def:lower bound unu} holds true with $\ul{m}:=\ul{r}_i$.

}
\end{rem}

	\begin{lem}[Weighted trace inequality for $h-h_{\Si\cap\Om}$ and $\na h$]
		\label{lem:weighted trace inequality ad hoc}
		For any $z \in \RR^N$ satisfying \eqref{eq:INTRO_inner product z in cone}, consider $h=q-u$ defined as in \eqref{eq:INTRO def h}.
		Let $\lbunu$ be the lower bound defined in \eqref{def:lower bound unu}.
		We have that
		\begin{equation}\label{eq:weighted trace for h-hOm}
			\nr h - h_{\Si\cap\Om}  \nr^2_{L^2(\Ga_0)} \le \frac{2}{\lbunu} \left( \frac{N}{\mu_{2,1}(\Si\cap\Om)^2}  +1 \right)   \, \nr (-u)^{ \frac{1}{2} } \na h \nr^2_{L^2(\Si\cap\Om)}  ,
		\end{equation}
		where $\mu_{2,1}(\Si\cap\Om)$ is the best constant in the Poincar\'e inequality \eqref{eq:BoasStraube-poincare} (with $p=2$, $\al=1$).
		
		Moreover, we have that
		\begin{equation}\label{eq:weighted trace for nah}
			\nr \na h \nr^2_{L^2(\Ga_0)} \le C \, \left( \nr (-u)^{ \frac{1}{2} }  \na^2 h \nr^2_{L^2(\Si\cap\Om)} + \int_{\Ga_1} u \langle  \na^2 u \na u , \nu \rangle dS_x \right) ,
		\end{equation}
		where the positive constant $C$ in \eqref{eq:weighted trace for nah} satisfies
		\begin{equation*}
			C \le \frac{2}{\lbunu} \left( N \, \La_{2,1}(k)^2  +1 \right),
		\end{equation*}
		where $\La_{2,1}(k)$ is the constant defined in \eqref{def:La p al k} (with $p=2$, $\al=1$).
	\end{lem}
	\begin{rem}
		{\rm Note that from the convexity of the cone, the error term 
		$$
		\int_{\Ga_1} u \langle  \na^2 u \na u , \nu \rangle dS_x
		$$ 
		is non-negative (see e.g. \cite[Formula (3.9)]{PT} for a proof). Such an error term will be re-absorbed later, as it appears in the left-hand side of the integral identity for Serrin's problem \eqref{idwps}.
	}
	\end{rem}
	\begin{proof}[Proof of Lemm \ref{lem:weighted trace inequality ad hoc}]
		Combining \eqref{eq:BoasStraube-poincare} (used here with $G:=\Si\cap\Om$, $p:=2$, $\al:=1$) and \eqref{eq:relation u dist general} we find that
		$$
		\int_{\Si\cap\Om} (h - h_{\Si\cap\Om})^2 dx \le 2 \, \mu_{2,1}(\Si\cap\Om)^{-2} \, \int_{\Si\cap\Om} (-u) |\na h |^2 dx .
		$$
		Putting together the last inequality, \cite[(5.12)]{Pog3}, 
		and \eqref{def:lower bound unu}, \eqref{eq:weighted trace for h-hOm} easily follows.

		Let us now prove \eqref{eq:weighted trace for nah}.
Combining Lemma \ref{lem:Mixed pp Poincareaigradienti} (used here with $p:=2$, $\al:=1$) and \eqref{eq:relation u dist general} we find that
		$$
		\int_{\Si\cap\Om} |\na h|^2 dx \le 2 \, \La_{2,1}(k)^2 \, \int_{\Si\cap\Om} (-u) |\na^2 h|^2 dx ,
		$$
		where $\La_{2,1}(k)$ is the constant defined in \eqref{def:La p al k} (with $p:=2$ and $\al:=1$).
		%
		The conclusion easily follows putting together the last inequality, \cite[(5.13)]{Pog3}, and \eqref{def:lower bound unu}.
	\end{proof}

The last Lemma that we need in order to prove Theorem \ref{thm:INTRO_Serrinstab Lipschitz} is the following
	
\begin{lem}
		Let $\lbunu$ be the lower bound defined in \eqref{def:lower bound unu}. We have that
		\begin{equation}\label{eq:Serrin Prestab nah}
			\nr \na h \nr_{L^2(\Ga_0)} \le \frac{C}{2} \, \nr u_\nu^2 -R^2 \nr_{L^2(\Ga_0)} ,
		\end{equation}
		where $C$ is the same constant appearing in \eqref{eq:weighted trace for nah}.
	\end{lem}
	\begin{proof}
		By putting together \eqref{eq:weighted trace for nah}, \eqref{eq:hessiana h},
		%
		%
		and \eqref{idwps}, we find that
		\begin{equation*}
			\nr \na h \nr^2_{L^2(\Ga_0)} \le \frac{C}{2} \, \int_{\Ga_0} \left( u_\nu^2 -R^2 \right) \, h_\nu \, dS_x ,
		\end{equation*}
		and, since by using H\"older's inequality we have that
		\begin{equation}\label{eq:RHSIDE Serrin}
			\int_{\Ga_0} \left( u_\nu^2 -R^2 \right) \, h_\nu \, dS_x 
			\le  \nr u_\nu^2 -R^2 \nr_{L^2(\Ga_0)} \nr h_\nu \nr_{L^2(\Ga_0)}
			\le \nr u_\nu^2 -R^2 \nr_{L^2(\Ga_0)} \nr \na h \nr_{L^2(\Ga_0)} ,
		\end{equation}
		the conclusion easily follows.
	\end{proof}

	We are now in position to prove Theorem \ref{thm:INTRO_Serrinstab Lipschitz}. 
	
	\begin{proof}[Proof of Theorem \ref{thm:INTRO_Serrinstab Lipschitz}]
		By using the triangle inequality, we compute:
		\begin{equation}\label{eq:triangle inequality for Serrin}
			\begin{split}
				\nr |x-z| - R \nr_{L^2(\Ga_0)}
				& \le \nr |x-z| - | \na u | \nr_{L^2(\Ga_0)} + \nr | \na u| - R  \nr_{L^2(\Ga_0)}
				\\
				& \le \nr  (x-z) -  \na u  \nr_{L^2(\Ga_0)} + \nr | \na u| - R  \nr_{L^2(\Ga_0)}
				\\
				& = \nr  \na h \nr_{L^2(\Ga_0)} + \nr u_\nu - R  \nr_{L^2(\Ga_0)} .
			\end{split}
		\end{equation}
		Estimating the first summand by using \eqref{eq:Serrin Prestab nah}, and the second summand by using that
		$$
		|u_\nu - R| \le \frac{1}{ \lbunu + R} \, | u_\nu^2 - R^2| \le \frac{1}{2 \, \lbunu } \, | u_\nu^2 - R^2| ,
		$$
		the conclusion easily follows. In the last inequality we used that $R=(u_\nu)_{\Ga_0} \ge \ul{m}$.
	\end{proof}
	
	\section{Stability estimates in terms of \texorpdfstring{$\rho_e - \rho_i$: proof of Theorem \ref{thm:Serrin general stability rhoe rhoi}}{rhoe -rhoi}}\label{sec:stability rhoe-rhoi}

	
In this section we use the same notations as in Section \ref{sec:sharp stability}.

	\begin{thm}

		\label{thm:Serrin-W22-stability in cones}
		Let $\Si\cap\Om$ be a bounded domain satisfying the $(\te,\ca)$-uniform interior cone condition.
		Let $z \in \RR^N$ be the point chosen as in \eqref{eq:INTRO_choice of z 1of2}-\eqref{eq:INTRO_choice of z 2of2}.

		Then, there exists an explicit positive constant $C$ such that
		$$
		\label{ineq:weighted-diff-radii-hessian}
		\rho_e - \rho_i \le  
		C \,
		\begin{cases}
			\displaystyle \nr \de_{\Ga_0} \na^2 h \nr_{L^2 (\Si\cap\Om)}  \max \left[ \log \left(  \frac{  e \, \nr \na h \nr_{L^\infty (\Si\cap\Om)} }{  \nr \de_{\Ga_0} \na^2 h \nr_{L^2(\Si\cap\Om)} } \right) , 1\right], \ &\mbox{for $N=2$};
			\vspace{3pt}
			\\
			\nr \na h \nr_{L^\infty (\Si\cap\Om)}^{ \frac{N-2}{N} }\,  \nr \de_{\Ga_0} \na^2 h \nr_{L^2 (\Si\cap\Om)}^{ \frac{2}{N}} ,   \ &\mbox{for $N\ge 3$.}
		\end{cases}
		$$
		The constant $C$ can be explicitly estimated only in terms of $N, \ca , \te$, the constant $\eta_{2, 1}(\Ga_1,\Si\cap\Om)$ from Theorem \ref{thm:Poincare new in general}, and the diameter $d_{\Si\cap\Om}$.
	\end{thm}
	\begin{proof}
		In what follows, we use the letter $C$ to denote a constant whose value may change line by line. All the constants $C$ can be explicitly computed (by following the steps of the proof) and estimated in terms of the parameters declared in the statement only (by recalling Remark \ref{rem:stime mu HS}).

		(i) Let $N=2$.
		We use \cite[Lemma 6.4]{Pog3} with $p:=N=2$ and get:
		$$
		\rho_e - \rho_i  \le 
		C \,\max \left\lbrace \nr \na h \nr_{L^2 (\Si\cap\Om)}  \log \left(  \frac{  e \, \nr \na h \nr_{L^\infty (\Si\cap \Om)} }{  \nr  \na h \nr_{L^2 (\Si\cap\Om)} } \right) , \nr  \na h \nr_{L^2 (\Si\cap \Om)} \right\rbrace.   
		$$
		Next, Lemma \ref{lem:Mixed pp Poincareaigradienti} with $p:=2$ and $\al:=1$ gives:
		$$
		\nr \na h \nr_{L^2 (\Si\cap\Om)} \le C  \, \nr \de_{\Ga_0} \na^2 h \nr_{L^2 (\Si\cap\Om)}.
		$$
		Thus, the desired conclusion ensues by invoking the monotonicity of the function 
		$t\mapsto t \max \{\log( A/t), 1\}$ for every $A>0$.
		
		(ii) When $N\ge 3$, we can use \cite[Lemma 6.4]{Pog3} with $p:=2$ and put it together with Lemma \ref{lem:Mixed pp Poincareaigradienti} with $p:=2$ and $\al:=1$.
	\end{proof}
	
	By coupling the previous theorem with a suitable upper bound for $\nr \na h \nr_{L^{\infty}(\Si\cap\Om)}$, we easily obtain the following.
	
	\begin{cor}
		\label{cor:Serrin-W22-stability in cones with upper bound}
		Let $\Si\cap\Om$ be a bounded domain satisfying the $(\te,\ca)$-uniform interior cone condition.
		Let $z \in \RR^N$ be the point chosen as in \eqref{eq:INTRO_choice of z 1of2}-\eqref{eq:INTRO_choice of z 2of2}.

		Then, there exists an explicit positive constant $C$ such that
		\begin{equation*}
			\rho_e - \rho_i \le  
			C \,
			\begin{cases}
				\displaystyle \nr \de_{\Ga_0} \na^2 h \nr_{L^2 (\Si\cap\Om)}  \max \left[ \log \left(  \frac{  e }{  \nr \de_{\Ga_0} \na^2 h \nr_{L^2(\Si\cap\Om)} } \right) , 1\right], \ &\mbox{for $N=2$};
				\vspace{3pt}
				\\
				\nr \de_{\Ga_0} \na^2 h \nr_{L^2 (\Si\cap\Om)}^{ \frac{2}{N}} ,   \ &\mbox{for $N\ge 3$.}
			\end{cases}
		\end{equation*}
		The constant $C$ can be explicitly estimated only in terms of $N, \ca , \te$, the constant $\eta_{2, 1}(\Ga_1,\Si\cap\Om)$ from Theorem \ref{thm:Poincare new in general}, the diameter $d_{\Si\cap\Om}$, and $\nr \na u \nr_{L^{\infty}(\Si\cap\Om)}$.
	\end{cor}
	\begin{proof}
		The proof is analogous to that of \cite[Corollary 6.7]{Pog3} with the only difference that to obtain the upper bound for $\nr \na h \nr_{L^\infty(\Si\cap\Om)}$ we now use \eqref{eq:dimprovaPoincaregradmixed} with $ \al:=1$ and $p:=2$ (instead of $\al:=0$ and $p:=2$), hence obtaining \cite[(6.8)]{Pog3} with $\eta_{2, 0}(\Ga_1,\Si\cap\Om)$ and $\nr \na^2 h \nr_{L^2 (\Si\cap\Om)}$ replaced by $\eta_{2, 1}(\Ga_1,\Si\cap\Om)$ and $\nr \de_{\Ga_0} \na^2 h \nr_{L^2 (\Si\cap\Om)}$.
	\end{proof}
	
	We are now in position to prove Theorem \ref{thm:Serrin general stability rhoe rhoi}.

	\begin{proof}[Proof of Theorem \ref{thm:Serrin general stability rhoe rhoi}]
		By putting together \eqref{eq:hessiana h} and \eqref{idwps} we find that
		\begin{equation*}
			\int_{\Si\cap\Om} (-u) | \na^2 h |^2 \, dx + \int_{\Ga_1} u \,  \langle \na^2 u \na u, \nu \rangle \, dS_x =
			\frac{1}{2}\,\int_{\Ga_0} \left( u_\nu^2- R^2\right) h_\nu \,dS_x ,
		\end{equation*}
		Discarding the second summand in the left-hand side (which is non-negative) and using \eqref{eq:RHSIDE Serrin} and \eqref{eq:Serrin Prestab nah} to estimate the right-hand side, we obtain that
		\begin{equation}\label{eq:stima prefinal Serrin rhoerhoi}
			\int_{\Si\cap\Om} (-u) | \na^2 h |^2 \, dx \le C \, \nr u_\nu^2 - R^2 \nr^2_{L^2(\Ga_0)} .
		\end{equation}
		The conclusion follows by putting together the last inequality, \eqref{eq:relation u dist general} and Corollary \ref{cor:Serrin-W22-stability in cones with upper bound}.
	\end{proof}
	
	The stability profile obtained in Theorem \ref{thm:Serrin general stability rhoe rhoi} can be improved whenever \eqref{eq:relation u dist general} can be replaced with the finer estimate \eqref{eq:relation u dist improved} relating $u$ and $\de_{\Ga_0}$. 
	%
	%
	To this aim, we will use the following strengthened version of Lemma \ref{lem:Mixed pp Poincareaigradienti} in the case where $p(1 -\al)<N$.
	\begin{lem}\label{lem:Mixed Strengthened Poincareaigradienti}
		Let $z \in \RR^N$ be the point chosen as in \eqref{eq:INTRO_choice of z 1of2}-\eqref{eq:INTRO_choice of z 2of2}.
		
		If $r, p, \al$ are as in \eqref{eq:r p al in Hurri},
		then we have that
		\begin{equation*}
			\nr \na h \nr_{L^{r}( \Si\cap\Om )} \le C \, \nr \de_{\Ga_0}^{\al} \, \na^2 h  \nr_{L^p ( \Si\cap\Om)} ,
		\end{equation*}
		for some positive constant $C$ satisfying $C \le \La_{r,p,\al} (k)$, where we have set
		\begin{equation}\label{eq:NEWCONSTANT La_rpal (k)}
		\La_{r,p,\al} (k) := 
		\begin{cases}
			\mu_{r ,p, \al}(\Si\cap\Om)^{-1} \quad & \text{if } k=0
			\\
			\eta_{r , p, \al} (\Ga_1,\Si \cap \Om )^{-1} \quad & \text{if } k=N
			\\
			\max\left[ \mu_{r , p,\al}(\Si\cap\Om)^{-1} , \, \eta_{ r , p, \al} (\Ga_1,\Si \cap \Om )^{-1}  \right] \quad & \text{if } 1 \le k \le N-1 ,
		\end{cases}
		\end{equation}
		where $\mu_{r, p,\al}(\Si\cap\Om)$ and $\eta_{r, p, \al} (\Ga_1,\Si \cap \Om )$ are those in \eqref{eq:John-Hurri-poincare} and Theorem \ref{thm:Strengthened Poincare new RN}.
	\end{lem}
	\begin{proof}
		In light of \eqref{eq:gradienteproiettatosuspan nu}, we can apply \eqref{eq:Strengthened Poincare new RN} in Theorem \ref{thm:Strengthened Poincare new RN} with $G:=\Si\cap\Om$, $A:=\Ga_1$, and $\vV:= (h_1, \dots, h_k , 0 , \dots, 0) $ to get that
		\begin{equation}\label{eq:dim prova Strengthened Poincaregradmixed}
			\left( \sum_{i=1}^k \nr h_i \nr_{L^{r} (\Si\cap\Om)} \right)^{1/r}\le \eta_{r, p , \al}(\Ga_1,\Si\cap\Om)^{-1} \, \left( \sum_{i=1}^k \sum_{j=1}^N \nr \de_{\Ga_0 }^{\al} \, h_{ij} \nr^p_{L^p (\Si\cap\Om)} \right)^{1/p}.
		\end{equation}
		In light of \eqref{eq:gradienteparte ortogonale a span nu}, we can apply \eqref{eq:John-Hurri-poincare} (with $G:=\Si\cap\Om$) to each first partial derivative $h_i$ of $h$, $i=k+1, \dots, N$.
		Raising to the power of $r$ those inequalities and \eqref{eq:dim prova Strengthened Poincaregradmixed}, and then summing up, the conclusion easily follows by using the inequality 	\begin{equation}\label{eq:ineq for Sobolev norm equivalence}
			\sum_{i=1}^N x_i^{\frac{r}{p}} \le \left( \sum_{i=1}^N x_i \right)^{\frac{r}{p}} ,
		\end{equation}
		which holds for every $(x_1, \dots, x_N) \in \RR^N$ with $x_i \ge 0$ for $i=1, \dots, N$, since $r/p \ge 1$.
	\end{proof}

	\begin{thm}
		\label{thm:Serrin-IMPROVED-W22-stability in cones}
		Let $\Si\cap\Om$ be a bounded domain satisfying the $(\te,\ca)$-uniform interior cone condition.
		Let $z \in \RR^N$ be the point chosen as in \eqref{eq:INTRO_choice of z 1of2}-\eqref{eq:INTRO_choice of z 2of2}. Then, there exists an explicit positive constant $C$ such that
		$$
		\label{ineq:weighted-IMPROVED-diff-radii-hessian}
		\rho_e - \rho_i  
		\le 
		C \,
		\begin{cases}
			\nr \de_{\Ga_0}^{1/2} \na^2 h \nr_{L^2(\Si\cap\Om)}   \ &\mbox{if $N=2$}; 
			\\
			\displaystyle  \nr \de_{\Ga_0}^{1/2} \na^2 h \nr_{L^2 (\Si\cap\Om)} \max \left[ \log \left( \frac{e \, \nr \na h \nr_{ L^\infty (\si\cap\Om) }  }{  \nr \de_{\Ga_0}^{1/2} \na^2 h \nr_{L^2 (\Si\cap\Om)}  } \right) ,1 \right]  \ &\mbox{if $N=3$}; \vspace{3pt} 
			\\
			\nr \na h \nr_{L^\infty(\Si\cap\Om)}^{(N-3)/(N-1)} \nr \de_{\Ga_0}^{1/2} \na^2 h \nr_{L^2 (\Si\cap\Om)}^{2/(N-1)}   \ &\mbox{if $N \ge 4$.}
		\end{cases}
		$$
		The constant $C$ can be explicitly estimated only in terms of $N, \ca , \te$, the constant
		$\eta_{2, 1/2}(\Ga_1,\Si\cap\Om)$ from Theorem \ref{thm:Poincare new in general},
		and the diameter $d_{\Si\cap\Om}$.
	\end{thm}
	\begin{proof}
As usual, we use the letter $C$ to denote a constant whose value may change line by line. All the constants $C$ can be explicitly computed (by following the steps of the proof) and estimated in terms of the parameters declared in the statement only.

In particular, in the following proof we are going to apply Lemma \ref{lem:Mixed Strengthened Poincareaigradienti}, which introduces (for some choices of $r$, $p$, $\al$) the constant $\La_{r,p,\al}(k)$ defined in \eqref{eq:NEWCONSTANT La_rpal (k)}. Notice that $\La_{r,p,\al}(k)$ can be estimated in terms of $N, r, p , \ca, \te, d_{\Si\cap\Om}$, and, if $1 \le k \le N$, $\eta_{p, \al}(\Ga_1 , \Si \cap \Om )$.
In fact, $\mu_{r,p,\al}(\Si\cap\Om)$ (which appears in \eqref{eq:NEWCONSTANT La_rpal (k)} if $0 \le k \le N-1$) can be estimated in terms of $N, p , \ca, \te, d_{\Si\cap\Om}$ by recalling Remark \ref{rem:stime mu HS}.
Moreover, from the statement of Theorem \ref{thm:Strengthened Poincare new RN} (and recalling Remark \ref{rem:stime mu HS}) we have that $\eta_{r,p,\al}(\Ga_1,\Si\cap\Om)$ (which appears in \eqref{eq:NEWCONSTANT La_rpal (k)} if $1 \le k \le N$) can be estimated in terms of $N, p , \ca, \te, d_{\Si\cap\Om}$ and $\eta_{p,\al}(\Ga_1 , \Si \cap \Om )$.

		(i) Let $N=2$. By using \cite[Lemma 6.4]{Pog3} with $p:=4$ we have that
		$$
		\rho_e - \rho_i \le C \, \nr \na h \nr_{L^4 (\Si\cap \Om)} .
		$$
		By applying Lemma \ref{lem:Mixed Strengthened Poincareaigradienti} (with $r:=4$, $p:=2$, and $\al:= 1/2$) we obtain that
		\begin{equation*}
			\nr \na h \nr_{L^4 (\Si\cap\Om)} \le C \, \nr \de_{\Ga_0}^{1/2} \na^2 h  \nr_{L^2 (\Si\cap\Om)} ,
		\end{equation*}
		and the conclusion follows.
		
		(ii) Let $N=3$. By using Lemma \ref{lem:Mixed Strengthened Poincareaigradienti} with $r:=3$, $p:=2$, $\al:= 1/2$, we get
		$$
		\nr \na h \nr_{L^3 (\Si\cap\Om)} \le C \, \nr \de_{\Ga_0}^{1/2} \na^2 h \nr_{L^2 (\Si\cap\Om)}.
		$$
		The conclusion follows by using \cite[Lemma 6.4]{Pog3} with $p:=N=3$.
		
		(iii) When $N\ge 4$, we use \cite[Lemma 6.4]{Pog3} with $p:=2N/(N-1)$ and put it together with Lemma \ref{lem:Mixed Strengthened Poincareaigradienti} with $r:={\frac{2N}{N-1}}$, $p:=2$, $\al:=1/2$.
	\end{proof}

	By coupling the previous theorem with a suitable upper bound for $\nr \na h \nr_{L^{\infty}(\Si\cap\Om)}$, we easily obtain the following.
	
	\begin{cor}
		\label{cor:Serrin-IMPROVED-W22-stability in cones with upper bound}
		Let $\Si\cap\Om$ be a bounded domain satisfying the $(\te,\ca)$-uniform interior cone condition.
		Let $z \in \RR^N$ be the point chosen as in \eqref{eq:INTRO_choice of z 1of2}-\eqref{eq:INTRO_choice of z 2of2}. Then, there exists an explicit positive constant $C$ such that
		$$
		\label{ineq:weighted-IMPROVED-diff-radii-hessian with gradient bound}
		\rho_e - \rho_i  
		\le 
		C \,
		\begin{cases}
			\nr \de_{\Ga_0}^{1/2} \na^2 h \nr_{L^2(\Si\cap\Om)}   \ &\mbox{if $N=2$}; 
			\\
			\displaystyle  \nr \de_{\Ga_0}^{1/2} \na^2 h \nr_{L^2 (\Si\cap\Om)} \max \left[ \log \left( \frac{e  }{  \nr \de_{\Ga_0}^{1/2} \na^2 h \nr_{L^2 (\Si\cap\Om)}  } \right) ,1 \right]  \ &\mbox{if $N=3$}; \vspace{3pt} 
			\\
			\nr \de_{\Ga_0}^{1/2} \na^2 h \nr_{L^2 (\Si\cap\Om)}^{2/(N-1)}   \ &\mbox{if $N \ge 4$.}
		\end{cases}
		$$
		The constant $C$ can be explicitly estimated only in terms of $N, \ca , \te$, the constant
		$\eta_{2, 1/2}(\Ga_1,\Si\cap\Om)$ from Theorem \ref{thm:Poincare new in general}, the diameter $d_{\Si\cap\Om}$, and, if $N \ge 3$, $\nr \na u \nr_{L^{\infty}(\Si\cap\Om)}$.
	\end{cor}
	\begin{proof}
		The proof is analogous to that of \cite[Corollary 6.7]{Pog3} with the only difference that to obtain the upper bound for $\nr \na h \nr_{L^\infty(\Si\cap\Om)}$ we now use \eqref{eq:dimprovaPoincaregradmixed} with $ \al:=1/2$ and $p:=2$ (instead of $\al:=0$ and $p:=2$), hence obtaining \cite[(6.8)]{Pog3} with $\eta_{2, 0}(\Ga_1,\Si\cap\Om)$ replaced by $\eta_{2, 1/2}(\Ga_1,\Si\cap\Om)$.
	\end{proof}
	
	We are now ready to prove the following improved version of Theorem \ref{thm:Serrin general stability rhoe rhoi}, under the additional geometrical assumption \eqref{eq:condition relation dist improved}.

	\begin{thm}[Improved stability in terms of $\rho_e - \rho_i$ for Serrin's problem in cones] 
		\label{thm:Improved Serrin stability rhoe rhoi}
		Let $\Si\cap\Om$ be a bounded domain and assume that $\Si$ is a convex cone and $\Si\cap\Om$ satisfies
		the $(\te,\ca)$-uniform interior cone condition.
		Assume that $\Si\cap\Om$ satisfies the $\ul{r}_i$-uniform interior sphere condition relative to the cone $\Si$ (as in Definition \ref{def:interior sphere relative to cone}) together with \eqref{eq:condition relation dist improved}. Let $z \in \RR^N$ be the point chosen in \eqref{eq:INTRO_choice of z 1of2}-\eqref{eq:INTRO_choice of z 2of2}.
		Then, we have that
		\begin{equation}
			\label{eq:stability Serrin rhoei Improved}
			\rho_e - \rho_i  \le 
			C \,
			\begin{cases}
				\nr u_\nu - R \nr_{L^2(\Ga_0)},  \ &\mbox{if } N =2
				\\
				\nr u_\nu - R \nr_{L^2(\Ga_0)} \max \left[ \log \left( \frac{1  }{ \nr u_\nu - R \nr_{L^2(\Ga_0)} } \right) , 1 \right],   \ &\mbox{if } N=3 ,
				\\
				\nr u_\nu - R \nr^{\frac{2}{N-1}}_{L^2(\Ga_0)},  \ &\mbox{if } N \ge 4 .
			\end{cases}
		\end{equation}
		The constant $C$ can be explicitly estimated only in terms of $N, \ca , \te$, the constant
		$\eta_{2, 1/2}(\Ga_1,\Si\cap\Om)$ from Theorem \ref{thm:Poincare new in general},
		the diameter $d_{\Si\cap\Om}$, $\ul{r}_i$, and, if $N \ge 3$, $\nr \na u \nr_{L^\infty(\Si\cap\Om)}$.
	\end{thm}
	\begin{proof}
		The conclusion follows by putting together \eqref{eq:stima prefinal Serrin rhoerhoi}, \eqref{eq:relation u dist improved} and Theorem \ref{thm:Serrin-IMPROVED-W22-stability in cones}.
	\end{proof}

	\begin{rem}\label{rem:stima norma na u infinito SicapOm con N r_e diam}
		{\rm
			Whenever $\Si\cap\Om$ satisfies the $\ul{r}_e$-uniform exterior sphere condition relative to $\Si$ in the sense of the definition introduced in \cite[Definition 4.6]{Pog3}, $\nr \na u \nr_{L^{\infty}(\Si\cap\Om)}$ can be explicitly estimated in terms of $N$, $d_{\Si\cap\Om}$ and $\ul{r}_e$ (see \cite[Lemma 4.7 and Lemma 4.8]{Pog3}).
		}
	\end{rem}

	\section{Additional remarks}\label{sec:additional remarks}
	\subsection{Alternative choices for the point \texorpdfstring{$z$}{z}}\label{subsec:alternative choices z}
	As already mentioned in the Introduction, different choices of the point $z$ lead to alternative stability results. For instance, we can avoid using \eqref{eq:Poincare new RN} (and \eqref{eq:Strengthened Poincare new RN}) and hence completely remove the dependence on $\eta_{p,\al}( \Ga_1 ,  \Si \cap \Om )^{-1}$ (and $\eta_{r,p,\al}( \Ga_1 ,  \Si \cap \Om )^{-1}$) for any $0 \le k \le N$, at the cost of leaving the point $z$ free to have non-zero components also in the directions spanned by $\nu$ on $\Ga_1$. A suitable choice to do this may be the following
	\begin{equation}\label{2nd choice z}
		z= \frac{1}{|\Si\cap\Om|} \int_{\Si\cap\Om} (x - \na u) \, dx .
	\end{equation}
	Thanks to this choice of $z$ we have the following result which is a modification of the results contained in Theorem \ref{thm:INTRO_Serrinstab Lipschitz} and in Theorems \ref{thm:Serrin general stability rhoe rhoi}, \ref{thm:Improved Serrin stability rhoe rhoi}.
	
	\begin{thm}\label{thm:final alternative z}
		Setting $z\in\mathbb{R}^N$ as in \eqref{2nd choice z} we have that:
\begin{itemize}
\item[(i)] Theorem \ref{thm:INTRO_Serrinstab Lipschitz} remains true with $\La_{2, 1}(k)$ replaced simply by $\mu_{2,1}(\Si\cap\Om)^{-1}$.

\item[(ii)] Theorem \ref{thm:Serrin general stability rhoe rhoi} holds true with an explicit constant $C$ only depending on $N$, $\ca$, $\te$, $d_{\Si\cap\Om}$, $\lbunu$, and $\nr \na u \nr_{L^\infty(\Si\cap\Om)}$. Moreover, Theorem \ref{thm:Improved Serrin stability rhoe rhoi}  holds true with an explicit constant $C$ only depending on $N$, $\ca$, $\te$, $d_{\Si\cap\Om}$, $\ul{r}_i$, and, if $N\ge 3$, $\nr \na u \nr_{L^\infty(\Si\cap\Om)}$.
\end{itemize}		
	\end{thm}

	\subsection{The classical case \texorpdfstring{$\Si=\RR^N$}{Sigma=RN}}
	\label{subsec:classical case Si=RN}

	The following theorem analyzes Theorems \ref{thm:INTRO_Serrinstab Lipschitz} and \ref{thm:Improved Serrin stability rhoe rhoi} in the particular case where $\Si=\RR^N$, returning variants of the results established in \cite{Feldman, MP3, MP6}.

	Notice that, when $\Si=\RR^N$, we have that $\Si\cap\Om=\Om$ is a smooth, say $C^2$, bounded domain in $\RR^N$.
	Such a domain always satisfies the classical uniform interior and exterior sphere conditions in $\RR^N$.
	Moreover, as already mentioned, when $\Si=\RR^N$ the uniform interior and exterior sphere conditions relative to $\Si$ reduce to the classical uniform interior and exterior sphere conditions in $\RR^N$.
	
	When $\Si=\RR^N$ the choice of $z$ in \eqref{eq:INTRO_choice of z 1of2} and \eqref{eq:INTRO_choice of z 2of2} agrees with that in \eqref{2nd choice z}, and reduces to the center of mass of $\Om$,
	being as
	$$
	z= \frac{1}{|\Om|} \int_{\Om} (x -\na u) \, dx  = \frac{1}{|\Om|} \left[ \int_{\Om} x  \, dx - \int_{\Ga_0} u \nu \, dx \right] =\frac{1}{|\Om|}  \int_{\Om} x  \, dx .
	$$
	We point out that, in the particular case $\Si=\RR^N$, many other choices for the point $z$ are admissible: we refer the interested reader to \cite{Feldman, MP, MP2, MP3}.
	
We are now in position to prove the following
	
		\begin{thm}[Sharp stability for the classical Serrin's problem in $\RR^N$]
		\label{thm:classical INTRO in R^N}
		Let $\Om \subset \RR^N$, $N\ge2$, be a bounded domain of class $C^2$. 
		Then, we have that:
		
		(i) \eqref{eq:INTRO_Serrin} holds true for an explicit constant $C$ only depending on $N$, $\ul{r}_i$, and $d_\Om$.
		
		(ii) \eqref{eq:stability Serrin rhoei Improved} holds true for an explicit constant $C$ only depending on $N$, $\ul{r}_i$, $\ul{r}_e$, and $d_\Om$. If $\Ga_0=\pa\Om$ is mean convex, then the dependence on $\ul{r}_e$ can be dropped.
	\end{thm}
	\begin{proof}[Proof of Theorem \ref{thm:classical INTRO in R^N}]
		As already noticed, (i) and (ii) immediately follow from Theorem \ref{thm:INTRO_Serrinstab Lipschitz} and Theorem \ref{thm:Improved Serrin stability rhoe rhoi} (recalling Remark \ref{rem:NEW on ADDITIONAL ASSUMPTION}). 
		
		Being as $\Si=\RR^N$ and hence $\Ga_1=\varnothing$, we have that $k=0$ and hence $\La_{p,\al}(0)= \mu_{p,\al}(\Si\cap\Om)^{-1}$ and $\La_{r,p,\al}(0)= \mu_{r,p,\al}(\Si\cap\Om)^{-1}$. In turn, being as $\Om$ a $C^2$ domain, $\mu_{p,\al}(\Si\cap\Om)^{-1}$ and $\mu_{r,p,\al}(\Si\cap\Om)^{-1}$ can be explicitly estimated in terms of $\ul{r}_i$ and $d_{\Si\cap\Om}$ only (see \cite[(iii) of Remark 2.4]{MP3}).

	Being as $\Si=\RR^N$, $\nr \na u \nr_{L^\infty(\Si\cap\Om)} = \nr \na u \nr_{L^\infty(\Om)}$ can be estimated by using \cite[Theorem 3.10]{MP}, which informs us that
	$$
	\nr \na u \nr_{L^\infty(\Om)} \le \frac{\max \left\lbrace 3, N \right\rbrace}{2} \frac{d_{\Om} (d_{\Om} + \ul{r}_e)}{\ul{r}_e} .
	$$ In the particular case where $\Ga_0=\pa\Om$ is mean convex, a better estimate is available, that is, $\nr \na u \nr_{L^{\infty}(\Om)}$ can be estimated in terms of $N$ and $\max_{\ol{\Om}}(-u)$ only (see, e.g, \cite[Lemma 2.2]{MP5}). In turn, $\max_{\ol{\Om}}(-u)$ can be easily estimated (e.g., applying \cite[(ii) of Lemma 4.9]{Pog3} in the special case $\Ga_1 = \varnothing$) by means of
	$$
	\max_{\ol{\Om}}(-u) \le \frac{d_{\Om}^2}{2} .
	$$
	We mention that a finer bound for $\max_{\ol{\Om}}(-u)$ in terms of
	the volume $|\Om|$ holds true thanks to a classical result on radially decreasing rearrangements due to Talenti (\cite{Talenti}).  
	\end{proof}

	\begin{rem}\label{rem:final remark on regularity in the classical setting}
	{\rm
	We recall that  the uniform interior and exterior touching ball condition is equivalent to the $C^{1,1}$ regularity of $\pa\Om$ (see, for instance, \cite[Corollary 3.14]{ABMMZ}). Nevertheless, (ii) of Theorem \ref{thm:classical INTRO in R^N} remains true by replacing $\ul{r}_i$, $\ul{r}_e$ (and hence relaxing the $C^{1,1}$ uniform regularity) with the $C^{1,\ga}$ regularity, for $0<\ga<1$: we refer to \cite{CPY} for details.
	For $ \ga \ge 1$ instead, the stability exponent in \eqref{eq:stability Serrin rhoei Improved} can be improved for $N\ge 4$, as proved in \cite[Theorem 4.4]{MP6}.
	}
	\end{rem}

\section*{Acknowledgements}
The authors are members of the Gruppo Nazionale Analisi Matematica Probabilit\`a e Applicazioni (GNAMPA) of the Istituto Nazionale di Alta Matematica (INdAM).
Giorgio Poggesi is supported by the Australian Research Council (ARC) Discovery Early Career Researcher Award (DECRA) DE230100954 ``Partial Differential Equations: geometric aspects and applications''
and is member of the Australian Mathematical Society (AustMS).

\end{document}